\documentclass[12pt]{article}

\usepackage[margin=1in]{geometry}
\usepackage{amsmath,amsfonts,amsthm,amssymb}
\usepackage{tikz}
\usepackage[shortlabels]{enumitem}
\usepackage{hyperref}
\usetikzlibrary{calc}

\theoremstyle{plain}
\newtheorem{thm}{Theorem}[section]
\newtheorem*{non-thm}{Theorem}
\newtheorem{lem}[thm]{Lemma}
\newtheorem{prop}[thm]{Proposition}
\newtheorem{cor}[thm]{Corollary}
\newtheorem{conj}[thm]{Conjecture}
\newtheorem{ques}[thm]{Question}

\theoremstyle{definition}
\newtheorem{df}[thm]{Definition}
\newtheorem*{non-df}{Definition}

\newtheorem{rmk}[thm]{Remark}

\newcommand{\comment}[1]{}

\newcommand{\abs}[1]{\left|#1\right|}

\newcommand{\ar}[1]{a_{#1}}
\newcommand{\An}{\mathbb{A}^N}
\newcommand{\bbN}{\mathbb{N}}
\newcommand{\bbP}{\mathbb{P}}
\newcommand{\bbR}{\mathbb{R}}

\newcommand{\calC}{\mathcal{C}}

\newcommand{\cone}{\textnormal{cone}}
\newcommand{\conv}{\textnormal{conv}}

\newcommand{\emp}{\emptyset}

\renewcommand{\int}[1]{\left[#1\right]}

\newcommand{\one}{\mathbf{1}}
\newcommand{\Par}[1]{\left(#1\right)}

\renewcommand{\phi}{\varphi}

\newcommand{\Rd}{\mathbb{R}^d}

\newcommand{\RN}{\mathbb{R}^N}

\newcommand{\set}[1]{\left\{#1\right\} }
\newcommand{\spa}{\text{span}}
\newcommand{\su}{\subseteq}
\newcommand{\sub}{\subset}

\numberwithin{equation}{section}

\newcounter{figure_num}

\begin{document}

\title{Radon Partitions of Random Gaussian Polytopes}
\author{Moshe J. White \thanks{Research supported in part by the Israel Science Foundation (Grant No. 2669/21) and in part by the ERC Advanced Grant (Grant No. 834735)} \\
Institute of Mathematics,\\
Hebrew University, Jerusalem, Israel}
\date{}
\maketitle
\begin{abstract}
In this paper we study a probabilistic framework for Radon partitions, where our points are chosen independently from the $d$-dimensional normal distribution. For every point set we define a corresponding Radon polytope, which encodes all information about Radon partitions of our set - with Radon partitions corresponding to faces of the polytope. This allows us to derive expressions for the probability that a given partition of $N$ randomly chosen points in $\mathbb{R}^d$ forms a Radon partition. These expressions involve conic kinematic formulas and intrinsic volumes, and in general require repeated integration, though we obtain closed formulas in some cases. This framework can provide new perspectives on open problems that can be formulated in terms of Radon partitions, such as Reay's relaxed Tverberg conjecture.
\end{abstract}

\section{Introduction}

For many topics in mathematics, it is often natural to gain insight by exploring random objects, and convex polytopes are no exception. There are a number of methods to produce random polytopes. One widely explored method is selecting vertices uniformly at random from a convex body. Questions regarding the distribution of properties of these random polytopes, such as the volume and face numbers, go as far back as a question asked by Sylvester in 1864. Another method produces a ``Gaussian polytope'', which is the convex hull of $N$ points independently using the $d$-dimensional standard normal distribution (see \cite{Rei} for a survey on random polytopes).

Radon's classic theorem states that any $x_1,\ldots,x_N\in\Rd$ may be partitioned into two sets with intersecting convex hulls, provided that $N\geq d+2$. We explore Radon partitions of the vertices of Gaussian polytopes -  vertices $x_1,\ldots,x_N\in\Rd$ chosen independently from a standard normal distribution. We succeed in principle in producing an expression for the probability that a given partition of such points is a Radon partition (depending on $N,d$, and the number of points in each part - see Proposition \ref{formula_gen}). The goal of this paper is to introduce this model as a natural random setting for exploring problems related to Radon's theorem.

As we explore in Section 2, all information about Radon partitions of a configuration of $N$ points in $\Rd$ is encoded in a ``Radon polytope'', which is obtained as the intersection between a fixed polytope (depending only on $N$) and a linear space $V$ (whose dimension depends on $N,d$). A partition of our points is a Radon partition if and only if $V$ intersects nontrivially with a certain cone (determined by the partition). In Section 3, we argue that vertices of the Gaussian polytope are a natural way to induce random Radon partitions, because in this setting the space $V$ is chosen \emph{uniformly} at random. We can therefore apply conic Kinematic formulas and obtain expressions for the probabilities that various Radon partitions will appear. In general these expressions involve repeated integration and thus rely on numeric approximation, but in some cases we obtain closed formulas. In Section 4 we give a very small example of open problems that can be described using Radon partitions.

For such problems we gain two new possible approaches: first, the problem is interpreted in terms of Radon polytopes. Second, we obtain a probabilistic version of our problem, which may be of independent interest.

\section{Radon polytopes}

In this paper, a polytope is the convex hull of a finite collection of points, and a cone is short for a convex polyhedral cone. We will make use of conic intrinsic volumes, and refer to \cite{AL} for all related definitions.\\
We will give a description of partitions of $N$ points using a polytope in real $(N-1)$-dimensional space, but the standard coordinate system for $\bbR^{N-1}$ isn't very convenient. Instead, we will use $N$ coordinates with a single dependency between them:

\begin{df}
For every $N\in\bbN$, denote
$$\An=\set{(\alpha_1,\ldots,\alpha_N)\in\RN\ :\ \sum_{i=1}^N \alpha_i=0}.$$
\end{df}
Informally, we consider $\An$ as the space of possible ``affine dependencies'' among $N$ points. Note that as a vector space, $\An$ is just $\bbR^{N-1}$. Throughout this paper, $N$ will correspond to the number of points to be partitioned, not the dimension of spaces.

\begin{df}
Let $e_1,\ldots,e_N$ denote the standard basis for $\RN$, and let $[N]$ denote the first $N$ natural numbers, $\set{1,2,\ldots,N}$. Given two disjoint subsets $A,B\sub[N]$, we define the \emph{partition cone} of $A$ and $B$ as
$$C(A,B)=\cone\set{e_i-e_j\ :\ i\in A,\ j\in B},$$
and the \emph{partition face} of $A$ and $B$ as
$$F(A,B)=\conv\set{e_i-e_j\ :\ i\in A,\ j\in B}.$$
\end{df}
We mention a basic property of partition cones:
\begin{lem} \label{partition_cone_lem}
Suppose that for every $i\in I$, we have a partition cone $C(A_i,B_i)$. Then
$$\bigcap_{i\in I} C(A_i,B_i)=C(\cap_i A_i, \cap_i B_i),$$
under the convention that $C(A,B)=\set{0}$ if $A=\emp$ or $B=\emp$.
\end{lem}
Given $\alpha=(\alpha_1,\ldots,\alpha_N)\in\An$, we have $\alpha\in C(A,B)$ if and only if $\alpha_i\geq 0$ for every $i\in A$, $\alpha_j\leq 0$ for every $j\in B$, and $\alpha_k=0$ for every $k\notin A\cup B$. We additionally have $\alpha\in F(A,B)$ if and only if 
$$\sum_{i\in A} \alpha_i=1=\sum_{j\in B} -\alpha_j\ .$$

Recall the following proof of Radon's classic theorem: if $x_1,\ldots,x_N\in\Rd$ and $N\geq d+2$, the points are affinely dependent. That is, there exists $0\neq\alpha\in\An$ such that $\sum_i \alpha_i x_i=0$ and $\sum_i \alpha_i=0$. Then $\alpha\in C(A,B)$ for some disjoint $A,B\sub[N]$: $A$ (and $B$) are simply the sets of indices of negative (and positive) coordinates of $\alpha$. W.l.o.g. we may assume that $\alpha$ is rescaled to satisfy $\alpha\in F(A,B)$. Now we rearrange the affine dependence as
$$\sum_{i\in A} \alpha_i x_i=\sum_{j\in B} -\alpha_j x_j\ ,$$
and this proves that $\conv\set{x_i:i\in A}\cap\conv\set{x_j:j\in B}\neq\emp$. We call this a Radon partition for $x_1,\ldots,x_N$:
\begin{df}
A partition in $[N]$ is an ordered pair $(A,B)$, where $A,B\sub[N]$ are disjoint. Such $(A,B)$ is a Radon partition for $x_1,\ldots,x_N\in\Rd$, if
$$\conv\set{x_i\ :\ i\in A}\cap\conv\set{x_j\ :\ j\in A}\neq\emp.$$
\end{df}
\begin{rmk}
We refer to a partition \emph{in} $[N]$, to emphasize that $A\cup B$ might be a proper subset of $[N]$.\\
We furthermore go against common terminology and regard a partition as an ordered pair - the first set specifies indices with positive values, and the second corresponds negative values. However the connection between $(A,B)$ and $(B,A)$ should be clear, e.g. $C(B,A)=-C(A,B)$, a partition $(A,B)$ is a Radon partition if and only if $(B,A)$ is a Radon partition, etc.
\end{rmk}

The proof draws a strong connection between \emph{affine dependencies} and \emph{Radon partitions} among any configuration of $N$ points in any $\Rd$. This motivates the following definition: 
\begin{df} \label{complete_polytope}
Let $2\leq N\in\bbN$. The \emph{complete Radon polytope} for $N$ points is
$$P_N=\set{(\alpha_1,\ldots,\alpha_N)\in\An\ :\ \sum_{i=1}^N \abs{\alpha_i}\leq 2}.$$
\end{df}

Note that $P_N$ is a polytope of dimension $N-1$, and the proper faces of $P_N$ are precisely the partition faces $\set{F(A,B)}$ (we may take $A=B=\emp$ to obtain the empty face). See Figure \ref{radon_polytopes} for the polytopes $P_3$ and $P_4$.

\begin{center}
\refstepcounter{figure_num}\label{radon_polytopes}

\begin{tikzpicture}

\coordinate (p1_2) at (0:2.5);
\filldraw (p1_2) circle (2pt) node[right] {$12$};
\coordinate (p1_3) at (60:2.5);
\filldraw (p1_3) circle (2pt) node[above right] {$13$};
\coordinate (p2_3) at (120:2.5);
\filldraw (p2_3) circle (2pt) node[above left] {$23$};
\coordinate (p2_1) at (180:2.5);
\filldraw (p2_1) circle (2pt) node[left] {$21$};
\coordinate (p3_1) at (240:2.5);
\filldraw (p3_1) circle (2pt) node[below left] {$31$};
\coordinate (p3_2) at (300:2.5);
\filldraw (p3_2) circle (2pt) node[below right] {$32$};

\draw (p1_2) -- (p1_3) -- (p2_3) -- (p2_1) -- (p3_1) -- (p3_2) -- (p1_2);

\begin{scope}[shift={(7.5,0)}]\scriptsize

\coordinate (p1_2) at (0:2.5);
\filldraw (p1_2) circle (2pt) node[right] {$12$};
\coordinate (p1_3) at (60:2.5);
\filldraw (p1_3) circle (2pt) node[above right] {$13$};
\coordinate (p2_3) at (120:2.5);
\filldraw (p2_3) circle (2pt) node[above left] {$23$};
\coordinate (p2_1) at (180:2.5);
\filldraw (p2_1) circle (2pt) node[left] {$21$};
\coordinate (p3_1) at (240:2.5);
\filldraw (p3_1) circle (2pt) node[below left] {$31$};
\coordinate (p3_2) at (300:2.5);
\filldraw (p3_2) circle (2pt) node[below right] {$32$};

\draw (p1_2) -- (p1_3) -- (p2_3) -- (p2_1) -- (p3_1) -- (p3_2) -- (p1_2);

\coordinate (p1_4) at (30:1.443);
\filldraw (p1_4) circle (2pt) node[above left] {$14$};
\coordinate (p4_3) at (90:1.443);
\filldraw (p4_3) circle (2pt) node[above] {$43$};
\coordinate (p2_4) at (150:1.443);
\filldraw (p2_4) circle (2pt) node[above right] {$24$};
\coordinate (p4_1) at (210:1.443);
\filldraw (p4_1) circle (2pt) node[below right] {$41$};
\coordinate (p3_4) at (270:1.443);
\filldraw (p3_4) circle (2pt) node[below] {$34$};
\coordinate (p4_2) at (330:1.443);
\filldraw (p4_2) circle (2pt) node[below left] {$42$};

\draw (p4_1) -- (p4_2) -- (p4_3) -- (p4_1);
\draw[dashed] (p1_4) -- (p2_4) -- (p3_4) -- (p1_4);

\draw (p1_2) -- (p4_2) -- (p3_2);
\draw[dashed] (p2_1) -- (p2_4) -- (p2_3);
\draw (p1_3) -- (p4_3) -- (p2_3);
\draw[dashed] (p3_1) -- (p3_4) -- (p3_2);
\draw (p2_1) -- (p4_1) -- (p3_1);
\draw[dashed] (p1_3) -- (p1_4) -- (p1_2);

\end{scope}

\end{tikzpicture}
\footnotesize\\
\medskip
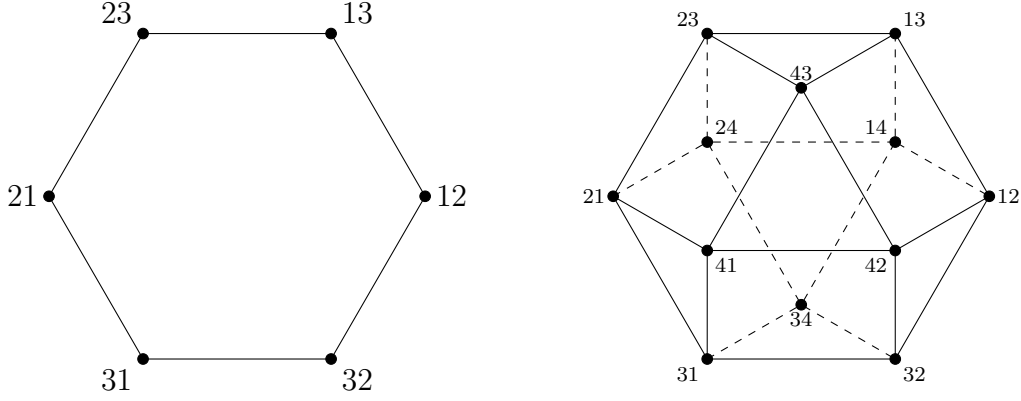
Figure \arabic{figure_num}: the polytopes $P_3$ (left) and $P_4$ (right). Each point labeled $ij$ corresponds to $e_i-e_j$.
\end{center}

\subsection{Radon polytopes corresponding to point collections}
We now define (general) Radon polytopes, which encode all relevant information about Radon partitions for some collection of points.
\begin{df}
Let $N,d\in\bbN$, and let $V\su\An$ be a linear subspace of codimension $d$ (i.e. dimension $N-d-1$). Then $P_N\cap V$ is called a \emph{Radon polytope} (for $N$ points in $\Rd$).

Let $x_1,\ldots,x_N\in\Rd$ be some collection of points, not contained in any proper affine subspace. We define the corresponding \emph{Radon polytope} as
$$P(x_1,\ldots,x_n)=\set{(\alpha_1,\ldots,\alpha_N)\in P_n\ :\ \sum_{i=1}^N \alpha_i x_i=0}.$$
\end{df}

Indeed, for a point collection $x_1,\ldots,x_N\in\Rd$, let $f:\RN\to\Rd$ denote the linear map taking each $e_i$ to $x_i$. Then $\ker(f)\su\RN$ is a subspace of dimension $N-d$ (unless the points are linearly dependent), and $V:=\ker(f)\cap\An$ is a subspace of dimension $N-d-1$ (unless the points lie on some affine subspace). Thus the $(N-d-1)$-dimensional polytope $P(x_1,\ldots,x_N)=P_N\cap V$, is a Radon polytope.

Any $\alpha\in V$ indicates a Radon partition of the collection, as in the proof of Radon's theorem above. Furthermore, $(A,B)$ is a Radon partition for $x_1,\ldots,x_N$ if and only if $F(A,B)\cap V\neq\emp$ (or equivalently, $C(A,B)\cap V\neq\set{0}$). Since the proper faces of $P_N\cap V$ are precisely those $F(A,B)\cap V$ which are non empty, there is a one to one correspondence between proper faces of the Radon polytope and the (ordered) Radon partitions of the points $x_1,\ldots,x_N$. \\
We summarize this as a Lemma:

\begin{lem} \label{radon_cone}
Let $x_1,\ldots,x_N\in\Rd$ be a collection of points, and let $A,B\sub[n]$ be disjoint subsets. 

Denote $V=\ker\Par{(a_1,\ldots,a_N)\mapsto\sum_i a_i x_i}\cap\An$. Then the following are equivalent:
\begin{enumerate}
\item $(A,B)$ is a Radon partition for $x_1,\ldots,x_N$,
\item $F(A,B)\cap V\neq\emp$,
\item $C(A,B)\cap V\neq\set{0}$.
\end{enumerate}
Furthermore, proper faces of the Radon polytope correspond to Radon partitions of $x_1,\ldots,x_N$.
\end{lem}

\subsubsection{Examples}
We consider a few simple examples of Radon polytopes:

\begin{enumerate}
\setcounter{enumi}{-1}
\item For any $N\in\bbN$, the complete Radon polytope $P_N$ (see Figure \ref{radon_polytopes}) is itself a Radon polytope for $N$ identical points, or $x_1=x_2=\cdots=x_N=0\in\bbR^0$ (take $\Rd$ where $d=0$).
\item If $N=d+2$ and $x_1,\ldots,x_{d+2}$ are $d+2$ points in general position in $\Rd$, then $V$ is one dimensional. $V$ intersects the boundary of $P_N$ at two points, $\alpha$ and $-\alpha$, which tell us the exact coefficients needed for a Radon combination (note that if $\alpha\in F(A,B)$ then $-\alpha\in -F(A,B)=F(B,A)$, so the two ordered Radon partitions are $(A,B)$ and $(B,A)$, usually regarded as a single Radon partition).
\item Four points on a line, for example $x_i=i\in\bbR$ for $i=1,2,3,4$. Since $N=4$ and $d=1$, the corresponding Radon polytope $P(1,2,3,4)$ has dimension $N-d-1=2$ (i.e. a polygon), shown in Figure \ref{line_4_points} below.
\begin{center}
\refstepcounter{figure_num}\label{line_4_points}

\begin{tikzpicture}

\coordinate (p13_2) at (2.449,-0.0);
\filldraw (p13_2) circle (2pt) node[right] {$13,2$};
\coordinate (p2_13) at (-2.449,0.0);
\filldraw (p2_13) circle (2pt) node[left] {$2,13$};
\coordinate (p14_2) at (2.177,1.217);
\filldraw (p14_2) circle (2pt) node[above right] {$14,2$};
\coordinate (p2_14) at (-2.177,-1.217);
\filldraw (p2_14) circle (2pt) node[below left] {$2,14$};
\coordinate (p14_3) at (-0.544,2.434);
\filldraw (p14_3) circle (2pt) node[above] {$14,3$};
\coordinate (p3_14) at (0.544,-2.434);
\filldraw (p3_14) circle (2pt) node[below] {$3,14$};
\coordinate (p24_3) at (-1.633,1.826);
\filldraw (p24_3) circle (2pt) node[above left] {$24,3$};
\coordinate (p3_24) at (1.633,-1.826);
\filldraw (p3_24) circle (2pt) node[below right] {$3,24$};

\draw (p13_2) -- (p14_2) -- (p14_3) -- (p24_3) -- (p2_13) -- (p2_14) -- (p3_14) -- (p3_24) -- (p13_2);

\begin{scope}[shift={(7.5,0)}]\scriptsize
\coordinate (p1_2) at (0:2.5);
\coordinate (p1_3) at (60:2.5);
\coordinate (p2_3) at (120:2.5);
\coordinate (p2_1) at (180:2.5);
\coordinate (p3_1) at (240:2.5);
\coordinate (p3_2) at (300:2.5);

\draw[gray] (p1_2) -- (p1_3) -- (p2_3) -- (p2_1) -- (p3_1) -- (p3_2) -- (p1_2);

\coordinate (p1_4) at (30:1.443);
\coordinate (p4_3) at (90:1.443);
\coordinate (p2_4) at (150:1.443);
\coordinate (p4_1) at (210:1.443);
\coordinate (p3_4) at (270:1.443);
\coordinate (p4_2) at (330:1.443);

\draw[gray] (p4_1) -- (p4_2) -- (p4_3) -- (p4_1);
\draw[gray, dashed] (p1_4) -- (p2_4) -- (p3_4) -- (p1_4);

\draw[gray] (p1_2) -- (p4_2) -- (p3_2);
\draw[gray, dashed] (p2_1) -- (p2_4) -- (p2_3);
\draw[gray] (p1_3) -- (p4_3) -- (p2_3);
\draw[gray, dashed] (p3_1) -- (p3_4) -- (p3_2);
\draw[gray] (p2_1) -- (p4_1) -- (p3_1);
\draw[gray, dashed] (p1_3) -- (p1_4) -- (p1_2);

\coordinate (p13_2) at (1.875,-1.083);
\filldraw (p13_2) circle (2pt) node[below right] {$\mathbf{13,2}$};
\coordinate (p2_13) at (-1.875,1.083);
\filldraw (p2_13) circle (2pt) node[above left] {$\mathbf{2,13}$};
\coordinate (p14_2) at (2.083,-0.241);
\filldraw (p14_2) circle (2pt) node[right] {$\mathbf{14,2}$};
\coordinate (p2_14) at (-2.083,0.241);
\filldraw (p2_14) circle (2pt) node[left] {$\mathbf{2,14}$};
\coordinate (p14_3) at (0.417,1.684);
\filldraw (p14_3) circle (2pt) node[above right] {$\mathbf{14,3}$};
\coordinate (p3_14) at (-0.417,-1.684);
\filldraw (p3_14) circle (2pt) node[below left] {$\mathbf{3,14}$};
\coordinate (p24_3) at (-0.625,1.804);
\filldraw (p24_3) circle (2pt) node[above left] {$\mathbf{24,3}$};
\coordinate (p3_24) at (0.625,-1.804);
\filldraw (p3_24) circle (2pt) node[below right] {$\mathbf{3,24}$};

\draw[very thick] (p13_2) -- (p14_2) -- (p14_3) -- (p24_3) -- (p2_13);
\draw[very thick, dashed] (p2_13) -- (p2_14) -- (p3_14) -- (p3_24) -- (p13_2);
\end{scope}

\end{tikzpicture}
\footnotesize\\
\medskip
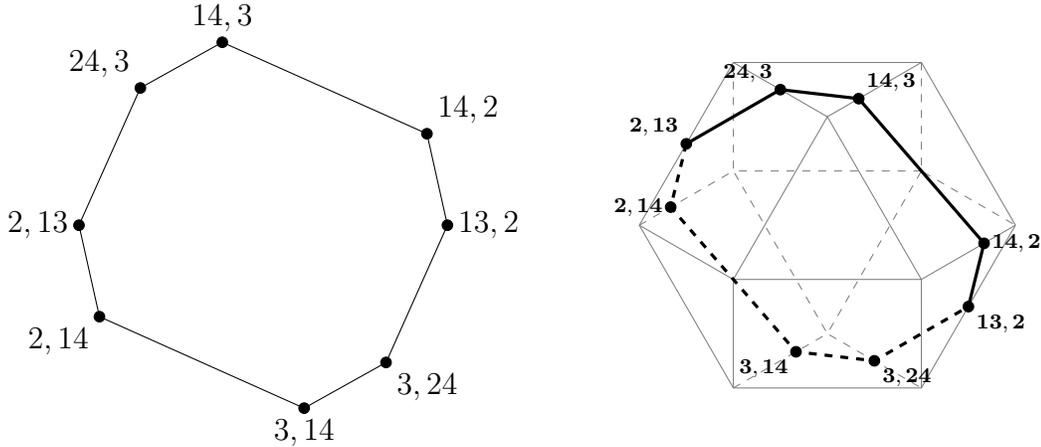
Figure \arabic{figure_num}: the polytope $P(1,2,3,4)$ (left), appearing as a subset of $P_4$ (right).
\end{center}

\item Five points on a line: $x_i=i$ for $i=1,2,3,4,5$. The corresponding $3$-dimensional Radon polytope is shown in Figure \ref{line_5_points}.

\begin{center}
\refstepcounter{figure_num}\label{line_5_points}
\begin{tikzpicture}[scale=0.8]

\coordinate (p13_2) at (-0.0,-1.378);
\filldraw (p13_2) circle (2pt) node[below right] {$13,2$};
\coordinate (p2_13) at (0.0,1.378);
\filldraw (p2_13) circle (2pt) node[above left] {$2,13$};
\coordinate (p14_2) at (3.043,-1.224);
\filldraw (p14_2) circle (2pt) node[below right] {$14,2$};
\coordinate (p2_14) at (-3.043,1.224);
\filldraw (p2_14) circle (2pt) node[above left] {$2,14$};
\coordinate (p15_2) at (2.282,0.663);
\filldraw (p15_2) circle (2pt) node[left] {$15,2$};
\coordinate (p2_15) at (-2.282,-0.663);
\filldraw (p2_15) circle (2pt) node[right] {$2,15$};
\coordinate (p14_3) at (6.086,0.306);
\filldraw (p14_3) circle (2pt) node[right] {$14,3$};
\coordinate (p3_14) at (-6.086,-0.306);
\filldraw (p3_14) circle (2pt) node[left] {$3,14$};
\coordinate (p15_3) at (4.564,4.081);
\filldraw (p15_3) circle (2pt) node[above right] {$15,3$};
\coordinate (p3_15) at (-4.564,-4.081);
\filldraw (p3_15) circle (2pt) node[below left] {$3,15$};
\coordinate (p15_4) at (-2.282,5.663);
\filldraw (p15_4) circle (2pt) node[above right] {$15,4$};
\coordinate (p4_15) at (2.282,-5.663);
\filldraw (p4_15) circle (2pt) node[below left] {$4,15$};
\coordinate (p24_3) at (4.564,0.918);
\filldraw (p24_3) circle (2pt) node[below left] {$24,3$};
\coordinate (p3_24) at (-4.564,-0.918);
\filldraw (p3_24) circle (2pt) node[above right] {$3,24$};
\coordinate (p25_3) at (3.043,3.639);
\filldraw (p25_3) circle (2pt) node[above] {$25,3$};
\coordinate (p3_25) at (-3.043,-3.639);
\filldraw (p3_25) circle (2pt) node[below] {$3,25$};
\coordinate (p25_4) at (-3.043,5.442);
\filldraw (p25_4) circle (2pt) node[above left] {$25,4$};
\coordinate (p4_25) at (3.043,-5.442);
\filldraw (p4_25) circle (2pt) node[below right] {$4,25$};
\coordinate (p35_4) at (-4.564,3.622);
\filldraw (p35_4) circle (2pt) node[above left] {$35,4$};
\coordinate (p4_35) at (4.564,-3.622);
\filldraw (p4_35) circle (2pt) node[below right] {$4,35$};

\draw[thick] (p13_2) -- (p14_2);
\draw[thick] (p13_2) -- (p15_2);
\draw[thick] (p13_2) -- (p3_24);
\draw[thick] (p13_2) -- (p3_25);
\draw[dashed] (p2_13) -- (p2_14);
\draw[dashed] (p2_13) -- (p2_15);
\draw[dashed] (p2_13) -- (p24_3);
\draw[dashed] (p2_13) -- (p25_3);
\draw[thick] (p14_2) -- (p15_2);
\draw[thick] (p14_2) -- (p14_3);
\draw[thick] (p14_2) -- (p4_25);
\draw[dashed] (p2_14) -- (p2_15);
\draw[dashed] (p2_14) -- (p3_14);
\draw[dashed] (p2_14) -- (p25_4);
\draw[thick] (p15_2) -- (p15_3);
\draw[thick] (p15_2) -- (p15_4);
\draw[dashed] (p2_15) -- (p3_15);
\draw[dashed] (p2_15) -- (p4_15);
\draw[thick] (p14_3) -- (p15_3);
\draw[dashed] (p14_3) -- (p24_3);
\draw[thick] (p14_3) -- (p4_35);
\draw[thick] (p3_14) -- (p3_15);
\draw[thick] (p3_14) -- (p3_24);
\draw[thick] (p3_14) -- (p35_4);
\draw[thick] (p15_3) -- (p15_4);
\draw[dashed] (p15_3) -- (p25_3);
\draw[thick] (p3_15) -- (p4_15);
\draw[thick] (p3_15) -- (p3_25);
\draw[thick] (p15_4) -- (p25_4);
\draw[thick] (p15_4) -- (p35_4);
\draw[thick] (p4_15) -- (p4_25);
\draw[dashed] (p4_15) -- (p4_35);
\draw[dashed] (p24_3) -- (p25_3);
\draw[dashed] (p24_3) -- (p4_35);
\draw[thick] (p3_24) -- (p3_25);
\draw[thick] (p3_24) -- (p35_4);
\draw[dashed] (p25_3) -- (p25_4);
\draw[thick] (p3_25) -- (p4_25);
\draw[thick] (p25_4) -- (p35_4);
\draw[thick] (p4_25) -- (p4_35);

\end{tikzpicture}
\footnotesize\\
\medskip
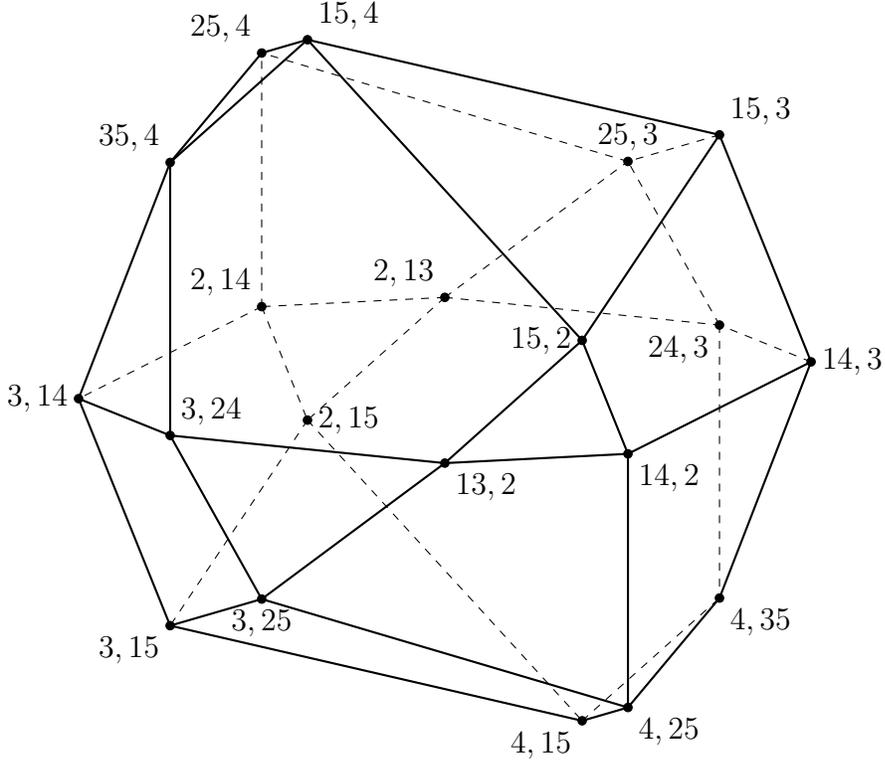
Figure \arabic{figure_num}: the polytope $P(1,2,3,4,5)$.
\end{center}

\end{enumerate}
The vertices of a Radon polytope correspond to \emph{minimal} Radon partitions in a collection of points for which it corresponds (also called primitive partitions, see \cite{HK}). In Figure \ref{line_4_points} (and other figures), $A,B$ denotes the point in $F(A,B)\sub P_N$ with the exact coordinates being those that appear in an affine dependency among points from the collection. We also use shorthand notation of $ij$ in place of $\set{i,j}$. So the vertex $14,3$ is located at $(\frac13, 0, -1, \frac23, 0)$, because $\frac13x_1 -x_3+\frac23x_4=0$, similarly $24,3$ is $(0, \frac12, -1, \frac12, 0)$, etc.

\section{Random polytopes}
So far, we have proved that every $N$ points $x_1,\ldots,x_N\in\Rd$ have a corresponding Radon polytope. We will soon see that the converse is also true: given a subspace $V\su\An$ of codimension $d$, there exists $N$ points in $\Rd$ for which $P_N\cap V$ is the corresponding Radon polytope, and describe how the subspace $V$ can be chosen randomly.

Suppose we have a question stated in terms of the Radon partitions of a collection or collections of $N$ points $X\sub\Rd$ (perhaps asking for the existence of $X$ with certain properties, or whether every such $X$ has certain properties). In this case, it is sufficient to understand the Radon polytope $P(X)=P_N\cap V$, where $V\su\An$ is a linear subspace of codimension $d$. It is often useful to ignore the specific collection $X$ altogether, and focus on the relevant subspace $V$ and Radon polytope $P_N\cap V$. If we choose $V$ randomly, we obtain a probabilistic version of our question.

One way to choose $V$ \emph{uniformly} at random, is by taking $x_1,\ldots,x_N\in\Rd$ to be the image of the standard basis of $\RN$, under a random orthogonal projection to $\Rd$. In more detail, choose $v_1,\ldots,v_{N-1}\in\An$ randomly (using any distribution which is invariant under orthogonal transformations). Let $w_0,w_1,\ldots,w_N\in\RN$ be the vectors obtained by applying the Gram-Schmidt process on $\one,v_1,\ldots,v_{N-1}$, where $\one=(1,\ldots,1)\in\RN$ is the vector with $1$'s in each coordinate. For each $i=1,\ldots,N$, let $e_i=\sum_j a^i_j w_j\in\RN$ be the $i$-th standard basis vector. Define $f:\RN\to\Rd$ as the linear map taking $e_i$ to $x_i:=\Par{a^i_{N-d},\ldots,a^i_{N-1}}$. Clearly, $\ker f=\spa\set{w_0,w_1,\ldots,w_{N-d-1}}$, therefore $\ker(f)\cap\An=\spa\set{w_1,\ldots,w_{N-d-1}}=\spa\set{v_1,\ldots,v_{N-d-1}}$. For such randomly chosen $x_1,\ldots,x_N\in\Rd$, we obtain the corresponding Radon configuration $P(x_1,\ldots,x_N)=P_N\cap V$, where $V$ is chosen uniformly at random.\\
Also note that for any fixed subspace $V\su\An$ of codimension $d$, we have $x_1,\ldots,x_N\in\Rd$ with $P(x_1,\ldots,X_N)=P_N\cap V$ (simply take $w_1,\ldots,w_{N-d-1}$ as a basis for $V$).

Perhaps surprisingly, we could also choose $x_1,\ldots,x_N$ independently from a Gaussian distribution. While the collections of points obtained in these two methods have quite different distributions (e.g. in the second method the points are independent), the corresponding random Radon configuration is uniformly distributed in both cases \cite{BV}.

Henceforth, whenever $x_1,\ldots,x_N\in\Rd$ and/or a subspace $V\su\An$ of codimension $d$ are chosen randomly, we assume that $V$ is chosen uniformly at random, and $x_1,\ldots,x_N$ are chosen independently from a standard Gaussian distribution. We also say that the Radon polytope $P_N\cap V$ is chosen uniformly at random.

\subsection{Probability of obtaining specific Radon partitions}
Consider a partition $[N]=A\cup B$, and denote $|A|=m,\ |B|=n$, so $N=m+n$. Let $\bbP_d(A,B)$ denote the probability that $(A,B)$ is a Radon partition for randomly chosen points $x_1,\ldots,x_N\in\Rd$ (equivalently, the probability that $F(A,B)$ intersects a randomly chosen Radon polytope for $N$ points in $\Rd$).

For a cone $C$ in some ambient space $\bbR^D$, let $v_0(C), v_1(C),\ldots,v_D(C)$ denote its conic intrinsic volumes (see \cite{AL} for definitions and any unproved statements regarding intrinsic volumes). Isometric cones have the same intrinsic volumes and in particular don't depend on the ambient space, so it is useful to follow the convention $v_k(C)\equiv 0$ for every $k>\dim(C)$. If $C$ is not a vector space, the probability that a random linear subspace $V\su\RN$ of codimension $d$ intersects $C$ nontrivially, is given by the conic Kinematic formulas (see \cite[\S 2.3.3]{AL}):
\begin{equation} \label{kinematic_formula}
\bbP(C\cap V\neq\set 0)=2\sum_{i\geq1\text{ odd}} v_{d+i}(C).
\end{equation}
Let $v_k(m,n)$ denote $v_k\Par{C(A,B)}$: this is well defined because any two partition cones $C(A,B)$ and $C(A',B')$ are isometric if $|A|=|A'|$ and $|B|=|B'|$. From Lemma \ref{radon_cone}, we obtain
\begin{equation} \label{radon_intrinsic}
\bbP_d(A,B)=2\sum_{i\geq1\text{ odd}} v_{d+i}(m,n),
\end{equation}
noting that $v_k(m,n)=0$ for $k\geq m+n$ (because $C(A,B)$ has dimension $|A|+|B|-1$). We also mention Gauss--Bonnet \cite[Corollary 4.4]{AL}, from which it follows that for $m,n\in\bbN$,
\begin{equation} \label{gauss-bonnet}
\sum_{i\text{ odd}} v_i(m,n)=\sum_{i\text{ even}} v_i(m,n)=\frac12.
\end{equation}

We will now explore $v_k(m,n)$ and $\bbP_d(A,B)$, for some values of $d,k,m,n$.

\subsubsection{The case $d=1$ and $k=0$}
\begin{prop}
Let $[m+n]=A\cup B$ be a partition with $|A|=m$ and $|B|=n$. Then
$$\bbP_1(A,B)=1-\frac{2}{\binom{m+n}{m}}.$$
\end{prop}
\begin{proof}
Let $x_1,\ldots,x_N\in\bbR$ be independent normal random variables, where $N=m+n$. The random configuration $P_N\cap\ker(f)$ is uniformly distributed, where $f:\RN\to\Rd$ takes each $e_i$ to $x_i$. With probability $1$, we have $x_{\sigma(1)}<x_{\sigma(2)}<\cdots<x_{\sigma(N)}$ for some permutation $\sigma:[N]\to[N]$. Furthermore, $\sigma$ is uniformly distributed, since $x_1,\ldots,x_N$ are chosen independently. Clearly $(A,B)$ is \emph{not} a Radon partition for $x_1,\ldots,x_N$ if and only if $x_i<x_j$ for every $i\in A, j\in B$ \emph{or} if $x_i>x_j$ for every $i\in A, j\in B$. Thus the probability that $(A,B)$ is not a Radon partition is $2\cdot{\binom{N}{m}}^{-1}$, and the claim follows.
\end{proof}
\begin{cor} \label{v_0}
From (\ref{radon_intrinsic}) and (\ref{gauss-bonnet}), we deduce that
$$v_0(m,n)=\binom{m+n}{m}^{-1}.$$
\end{cor}

\subsubsection{The case $m=1$}
Let $A=\set{n+1}$ and $B=[n]$. Note that $C=C(A,B)=\cone\set{e_{n+1}-e_i\ :\ i\in[n]}$, and the points $\set{e_{n+1}-e_i\ :\ i\in[n]}$ are the vertices of a regular simplex.  The intrinsic volumes of such cones was explored in \cite{KZ} by Kabluchko and Zaporozhets, who denoted $C_n(r)=\cone\set{u_1,\ldots,u_n}$, where $-1/n\leq r\in\bbR$ and $u_1,\ldots,u_n$ are linearly independent vectors (in some $\bbR^D$), satisfying
$$\left<u_i,u_j\right>=\begin{cases}
1+r &\text{ if } i=j,\\
r&\text{ if } i\neq j.\end{cases}$$
Then:

\begin{prop}[{\cite[Propositions 1.3--1.4]{KZ}}] \label{g_values}
Define $g_0:\bbR\to\bbR$ by $g_0(r)=1$. For every $n\in\bbN$, there exists a (differentiable) function $g_n:[-\frac 1n,\infty)\to[0,1]$, such that:
\begin{enumerate}\setcounter{enumi}{-1}
\item
$$v_k\Par{C_n(r)}=\binom{n}{k}g_k\Par{-\frac{r}{1+kr}}g_{n-k}\Par{\frac{r}{1+kr}}$$
for every $r>\frac1n$,
\item $g_1(r)=\frac 12$ ,
\item $g_2(r)=\frac14+\frac1{2\pi}\arcsin\frac{r}{1+r}$ ,
\item $g_3(r)=\frac18+\frac3{4\pi}\arcsin\frac{r}{1+r}$ ,
\item For $2\leq n\in\bbN$ we have
$$g_n\Par{-\frac1n}=0\ ,\quad g_n(0)=2^{-n}\ ,\quad g_n(1)=\frac1{n+1}\ ,$$
\item For $2\leq n\in\bbN$ and $r>-\frac1n$ we have
$$g'_n(r)=\frac{n(n-1)}{4\pi(r+1)\sqrt{(2r+1)}}g_{n-2}\Par{\frac r{2r+1}}\ .$$
\end{enumerate}
\end{prop}

In particular, in this case $C(A,B)$ is isometric to $C_n(1)$, so Proposition \ref{g_values} (0) becomes
\begin{equation} \label{formula_m1}
v_k(1,n)=\binom{n}{k}g_k\Par{-\frac{1}{k+1}}g_{n-k}\Par{\frac{1}{k+1}}.
\end{equation}

\subsubsection{The case $m=2$}
To obtain a formula for $v_k(m,n)$ when $m>1$ in terms of the functions $g_0,g_1,\ldots$, we will first prove a certain inclusion-exclusion principle for the intrinsic volumes. From (\ref{kinematic_formula}), we learn that for a cone $C\su\RN$ that is not a vector space,
$$v_k(C)=\frac12\Big(\bbP(C\cap V_{k-1}\neq\set0)-\bbP(C\cap V_{k+1}\neq\set0)\Big),$$
where each $V_d$ is a randomly chosen vector subspace of codimension $d$. Note that the above equation holds even when $C$ is a subspace of dimension $\ell\neq k,k+1$: in this case $v_{\ell}(C)=1$, $v_d(C)=0$ for $d\neq\ell$ (see \cite{AL}), and $\bbP(C\cap V_d)$ is $1$ if $\ell>d$ and $0$ otherwise. We deduce:

\begin{lem}[Inclusion-exclusion] \label{ie_lem}
Let $C,D$ be cones in $\RN$, such that $C\cup D$ is a cone. Then
$$v_k(C\cup D)=v_k(C)+v_k(D)-v_k(C\cap D),$$
provided that none of the cones $C\cup D,C,D,C\cap D$ are a linear subspace of dimension $k$ or $k+1$. More generally, if $C_1,\ldots,C_n$ and $\cup_i C_i$ are cones, then
\begin{equation} \label{IE}
v_k\Par{\bigcup_{i=1}^n C_i}=\sum_{j=1}^n (-1)^{j+1}\Par{\sum_{1\leq i_1<\cdots<i_j\leq n} v_k\Par{C_{i_1}\cap\cdots\cap C_{i_j}}},
\end{equation}
provided that none of the inputs to the intrinsic volume function $v_k(\cdot)$ that appear in the equation are a linear subspace of dimension $k$ or $k+1$.
\end{lem}
\begin{proof}
It is straightforward to verify that (\ref{IE}) holds if we replace $v_k$ with the function
$$\widetilde v_k(C):=\frac12\Big(\bbP(C\cap V_{k-1}\neq\set0)-\bbP(C\cap V_{k+1}\neq\set0)\Big),$$
which is defined on any cone $C$. If the assumptions of the lemma hold, $\widetilde v_k$ and $v_k$ are in agreement for every input appearing in the equation.
\end{proof}

Fix $n$, and let $D_i=C(\set i,[n+2]-\set{i})$ for $i=1,2$, and $D_{12}=C(\set{1,2},[n+2]-\set{1,2})$. Then $C=D_1\cup D_2\cup D_{12}$ is itself a cone, and $C=L\oplus C_0$, where $L=\spa\set{e_1-e_2}$, and
$$C_0=\cone\set{\frac12(e_1+e_2)-e_i\ :\ i=3,4,\ldots,n+2},$$
which is isometric to $C_n(\frac12)$. Recall that:

\begin{lem} \label{cone_sum}
Let $L\su\bbR^D$ be a linear subspace of dimension $\ell$, and let $C_0\su L^\perp$ be a cone. Then the intrinsic volumes of the cone $L+C_0$ are given by
$$v_k(L+C_0)=\begin{cases}
0& \text{ if } k<\ell,\\
v_{k-\ell}(C_0)& \text{ if } k\geq\ell.
\end{cases}$$
\end{lem}
This follows from the definition of the intrinsic volumes from \cite{AL} and the obvious bijection between $k$-dimensional faces of $L+C_0$ and $(k-\ell)$-dimensional faces of $C_0$.
 
In our case, $D_i\cap D_{12}=C(\set{i},[n+2]-\set{1,2})$, and other intersections among $D_1,D_2,D_{12}$ equal $\set{0}$ (see Lemma \ref{partition_cone_lem}).

In particular, Lemma \ref{ie_lem} holds for $k>0$, and we obtain
\begin{equation} \nonumber
\begin{split}
v_{k-1}\Par{C_n(\frac12)}=v_k(C)&=v_k(D_{12})+v_k(D_1)+v_k(D_2)-v_k(D_1\cap D_{12})-v_k(D_2\cap D_{12})\\
&=v_k(2,n) +2v_k(1,n+1)-2v_k(1,n).
\end{split}
\end{equation}
Applying (\ref{formula_m1}) and Proposition \ref{g_values} and simplifying, we obtain for $k=1,2,\ldots,n+1$:
\begin{equation} \label{formula_m2} \begin{split}
v_k(2,n)=&\binom{n}{k-1} g_{k-1}\Par{-\frac1{k+1}}g_{n-k+1}\Par{\frac1{k+1}}\\
&+2g_k\Par{-\frac1{k+1}}\left[\binom{n}{k}g_{n-k}\Par{\frac1{k+1}}-\binom{n+1}{k}g_{n-k+1}\Par{\frac1{k+1}}\right]
\end{split}\end{equation}
under the convention that $\binom{n}{\ell}=0$ for $\ell>n$.

\subsubsection{The case $m\geq 3$}
In principle, the method by which (\ref{formula_m2}) is obtained using inclusion-exclusion, can be applied for every $m$:
\begin{prop} \label{formula_gen}
For every disjoint $A,B\su[N]$ such that $|A|=m,\ |B|=n$ and every $d\in\bbN$, there exists a closed formula for $\bbP_d(A,B)$, consisting only of addition, subtraction, multiplication, integers, and the values $g_\ell\Par{\pm\frac1j}$ where $\ell=0,1,\ldots,m+n-1$ and $j\geq d+2$.
\end{prop}
\begin{proof}
By applying (\ref{radon_intrinsic}), it suffices to prove that for every $k,m,n\in\bbN$ such that $1\leq k\leq m+n-1$, there exists a closed formula for $v_k(m,n)$ in terms of the values $g_\ell\Par{\pm\frac1{k+1}}$, for $\ell=0,1,\ldots,m+n-1$. This is done by generalizing the proof leading to (\ref{formula_m2}), using inclusion-exclusion:

Let $m,n\in\bbN$ and consider the collection of cones
$$\calC_m= \set{C(A,[m+n]-A)\ :\ \emp\neq A\su[m]}.$$
Then $\cup\calC_m=L+C_0$, where
$$L=\spa\set{e_1-e_i\ :\ i=2,\ldots,m}$$
and
$$C_0=\cone\set{\frac1m(e_1+\ldots+e_m)-e_i\ :\ i=m+1,\ldots,m+n},$$
which is isometric to $C_n(\frac1m)$. Thus Lemma \ref{ie_lem} applies to $\cup\calC_m$, and (\ref{IE}) becomes 
\begin{equation} \label{sum_general_m}
v_k\Par{\cup\calC_m}=\sum_{\calC'} (-1)^{|\calC'|+1} v_k\Par{\cap\calC'},
\end{equation}
where the sum is extended over all non-empty subcollections $\calC'\su\calC_m$. From Lemma \ref{cone_sum}, $v_k\Par{\cup\calC_m}=v_k\Par{L+C_n(\frac1m)}$ is either $0$ (if $k<m-1$) or
$$v_k\Par{\cup\calC_m}=v_{k-m+1}\Par{C_n(\frac1m)},$$
which has a formula as required from Proposition \ref{g_values} (0). It follows from Lemma \ref{partition_cone_lem} that the right hand side of (\ref{sum_general_m}) consists of $v_k(m,n)$ and a combination of terms $v_k(m',n')$ with $m'<m$. Our claim now follows by induction on $m$.
\end{proof}
For $m=3$, we obtain for $k=1,2,\ldots,n+2$:
\begin{equation} \label{formula_m3} \begin{split}
v_k(3,n)=&\binom{n}{k-2}g_{k-2}^-g_{n-k+2}^+\\
&+3g_{k-1}^-\left[\binom{n}{k-1}g_{n-k+1}^+ -\binom{n+1}{k-1}g_{n-k+2}^+\right]\\
&+3g_k^-\left[\binom{n}{k}g_{n-k}^+ -2\binom{n+1}{k}g_{n-k+1}^+ +\binom{n+2}{k}g_{n-k+2}^+\right]
\end{split}\end{equation}
under the convention that $\binom{n}{\ell}=0$ for $\ell<0$ or $\ell>n$, where $g_\ell^\pm$ is short for $g_\ell\Par{\pm\frac1{k+1}}$.

\subsubsection{The case $k=m+n-1$}
While the general formula of $v_k(m,n)$ is quite complicated, we can simplify it when $k=m+n-1$:
\begin{prop} \label{formula_kmax}
Let $n,m\in\bbN$. Then
$$v_{m+n-1}(m,n)=\sum_{i=0}^{m-1}(-1)^i\binom{m}{i}g_{n+i}\Par{-\frac1{m+n}}.$$
\end{prop}

\begin{proof}
We start by applying the inclusion-exclusion method on $\calC_m$ from above, and show that (\ref{sum_general_m}) simplifies when $k=n+m-1$. Indeed, a partition cone $C(A,B)$ has dimension $|A|+|B|-1$. If $A\cup B$ and $A'\cup B'$ are two partitions of $[m+n]$, then $C(A,B)\cap C(A',B')=C(A\cap A',B\cap B')$ has dimension strictly less than $m+n-1$. Thus $v_{m+n-1}(\cap\calC')=0$ whenever $\calC'$ consists of more than one cone, and (\ref{sum_general_m}) for $k=m+n-1$ simplifies to
\begin{equation} \nonumber
\begin{split}v_{m+n-1}(\cup\calC_m)&=\sum_{\emp\neq A\su[m]} v_{m+n-1}\Par{C(A,[m+n]-A)}\\
&=\sum_{j=1}^m \binom{m}{j}v_{m+n-1}(j,m+n-j),
\end{split}\end{equation}
Proposition \ref{g_values} and Lemma \ref{cone_sum} allow us to substitute $v_{m+n-1}(\cup\calC_m)=v_n\Par{C_n(\frac1m)}=g_n\Par{-\frac1{m+n}}$, and rearranging gives us
\begin{equation} \label{sum_maxk}
v_{m+n-1}(m,n)=g_n\Par{-\frac1{m+n}}-\sum_{j=1}^{m-1} \binom{m}{j} v_{m+n-1}\Par{j,m+n-j}.
\end{equation}
We can now prove our claim by induction on $m$. For $m=1$, take (\ref{formula_m1}) with $k=n$. For $m>1$, take (\ref{sum_maxk}) and substitute each $v_{m+n-1}(j,m+n-j)$ using the induction hypothesis. The proof follows.
\end{proof}

\begin{cor} \label{curious_identity}
Clearly $v_k(m,n)=v_k(n,m)$, so proposition \ref{formula_kmax} implies the following curious identity: for every $n,m\in\bbN$, 
$$\sum_{i=0}^{m-1}(-1)^i\binom{m}{i}g_{n+i}\Par{-\frac1{m+n}}=\sum_{i=0}^{n-1}(-1)^i\binom{n}{i}g_{m+i}\Par{-\frac1{m+n}}.$$
\end{cor}

\subsection{Are balanced partitions more common?}
Considering partitions $[N]=A\cup B$, which are more likely to appear as Radon partitions in our random setting: those where one of $A,B$ is small and the other is large, or more ``balanced'' partitions, where $A,B$ have roughly the same number of points? Assume w.l.o.g. that $|A|\leq|B|$. We say that $(A,B)$ is balanced if $|A|=\lfloor\frac N2\rfloor$ (in the other extreme, we have $|A|=1$). There is some reason to suspect that balanced partitions are more common: if we take $N$ points in convex position, none of the Radon partitions $A,B$ have $|A|=1$. If our $N$ points are on the moment curve in $\Rd$, any Radon partition $(A,B)$ has $|A|\geq\lfloor\frac{d+2}{2}\rfloor$ (see \cite{Br} and \cite{HK}). On the other hand, every $N=d+3$ points in $\Rd$ have some balanced Radon partition (this is due to Flores and Van Kampen and holds even as a continuous analogue of Radon's theorem, see e.g. \cite[Theorem 6.1]{BFZ}, also see \cite{Sh}).

In this sense, balanced partitions are unavoidable for large $N$, while non-balanced partitions can be avoided. We conjecture that  for every $N,d$, the probability of finding a partition $[N]=A\cup B$ as a Radon partition among $N$ points in $\Rd$, is higher the more balanced our partition is:

\begin{conj} \label{balanced_prob}
Let $N,d\in\bbN$ such that $N\geq d+2$, and let $[N]=A\cup B=A'\cup B'$ be two partitions such that $|A|<|A'|\leq|B'|<|B|$. Then
$$\bbP_d(A,B)<\bbP_d(A',B').$$
\end{conj}
Note that we require $N\geq d+2$ because otherwise we have the trivial case $P_d(A,B)=0$.

\section{Open problems described by Radon partitions}

We consider two open problems around Radon partitions, and describe them as problems regarding Radon polytopes.

\subsection{Radon with tolerance}
Let $[N]=A\cup B$ be a Radon partition of some $x_1,\ldots,x_N\in\Rd$. We say that $A\cup B$ is a Radon partition \emph{with tolerance}, if $A\setminus\set{i}\cup B\setminus\set{i}$ is a Radon partition of $x_1,\ldots,x_{i-1},x_{i+1},\ldots,x_N$ for every $i\in[N]$. Let $t(d)$ denote the minimal $N$ such that \emph{any} collection of $N$ points in $\Rd$ admits a Radon partition with tolerance. Larman \cite{La} proved that $t(d)\leq 2d+3$, suggested that equality holds for every $d$, and proved it for $d\leq 4$. A lower bound of $t(d)\geq\lceil\frac{5d}3\rceil$ follows from a result by Ram\'irez-Alfons\'in \cite{RA}.

Clearly, $t(d)>N$ if and only if there exist some points $x_1,\ldots,x_N\in\Rd$ for which every Radon partition is \emph{not} a partition with tolerance (we may assume the points are in general position). We can describe partitions with tolerance in terms of the Radon configuration: we have $P:=P(x_1,\ldots,x_N)=P_N\cap V$, where $V\subset\An$ is a linear subspace of dimension $N-d-1$. From Lemma \ref{radon_cone}, we know that Radon partitions $[N]=A\cup B$ correspond to facets $F=F(A,B)\cap V$ of $P$ (that is, faces of dimension $N-d-2$). Every subface of $F$ is of the form $F(A',B')\cap V$, where $A'\su A$ and $B'\su B$. That is, subfaces of $F$ correspond to a sub-partitions of the points, and therefore $(A,B)$ is a Radon partition with tolerance if and only if $F$ contains $N$ ridges (subfaces of codimension $1$). To summarize, we obtain the following geometric interpretation for Radon partitions with tolerance:
\begin{ques}[Radon with tolerance]
Let $d\in\bbN$ and denote $N=2d+2$. Does there exist a linear subspace $V\subset\An$ of dimension $d+1$, such that whenever $[N]=A\cup B$ is a partition and $F=F(A,B)\cap V$ has dimension $d$, then $F$ has less than $N$ faces of dimension $d-1$.
\end{ques}
For every $d$, the answer to the above is positive if and only if $t(d)=2d+3$.

\subsubsection{Examples}
Consider two different configurations of $6$ points in $\bbR^2$: the first consists of $6$ different points on the unit circle, has a Radon partition with tolerance, and the corresponding Radon polytope has a hexagonal face (see Figure \ref{tolerance_fig} below). The second consists of $5$ points on the unit circle that form the vertices of a regular pentagon, along a sixth point at the origin. This configuration has no partition with tolerance. For the first configuration, the Radon configuration has a hexagonal face, but the polytope corresponding to the second configuration has only triangles, quadrilaterals, and pentagons (see Figure \ref{no_tolerance_fig} below)

\begin{center}
\refstepcounter{figure_num}\label{tolerance_fig}
\begin{tikzpicture}[scale=0.98]
\coordinate (v0) at (360 * 0 / 6:1);
\filldraw (v0) circle (2pt);
\coordinate (v1) at (360 * 1 / 6:1);
\filldraw (v1) circle (2pt);
\coordinate (v2) at (360 * 2 / 6:1);
\filldraw (v2) circle (2pt);
\coordinate (v3) at (360 * 3 / 6:1);
\filldraw (v3) circle (2pt);
\coordinate (v4) at (360 * 4 / 6:1);
\filldraw (v4) circle (2pt);
\coordinate (v5) at (360 * 5 / 6:1);
\filldraw (v5) circle (2pt);

\node at (360 * 0 / 6:1.3) {$x_1$};
\node at (360 * 1 / 6:1.3) {$x_2$};
\node at (360 * 2 / 6:1.3) {$x_3$};
\node at (360 * 3 / 6:1.3) {$x_4$};
\node at (360 * 4 / 6:1.3) {$x_5$};
\node at (360 * 5 / 6:1.3) {$x_6$};

\begin{scope}[shift={(8.5,0)}]\scriptsize
\coordinate (p13_24) at (5.0,0.0);
\coordinate (p24_13) at (-0.667,0.0);
\coordinate (p13_25) at (2.944,1.7);
\coordinate (p25_13) at (-1.0,-0.577);
\coordinate (p13_26) at (2.5,4.33);
\coordinate (p26_13) at (-0.333,-0.577);
\coordinate (p14_25) at (0.998,1.728);
\coordinate (p25_14) at (-0.998,-1.728);
\coordinate (p14_26) at (0.0,3.4);
\coordinate (p26_14) at (-0.0,-1.154);
\coordinate (p15_26) at (-2.5,4.33);
\coordinate (p26_15) at (0.333,-0.577);
\coordinate (p14_35) at (-0.0,1.154);
\coordinate (p35_14) at (-0.0,-3.4);
\coordinate (p14_36) at (-0.998,1.728);
\coordinate (p36_14) at (0.998,-1.728);
\coordinate (p15_36) at (-2.944,1.7);
\coordinate (p36_15) at (1.0,-0.577);
\coordinate (p15_46) at (-5.0,0.0);
\coordinate (p46_15) at (0.667,-0.0);
\coordinate (p24_35) at (-0.333,0.577);
\coordinate (p35_24) at (2.5,-4.33);
\coordinate (p24_36) at (-1.0,0.577);
\coordinate (p36_24) at (2.944,-1.7);
\coordinate (p25_36) at (-1.996,0.0);
\coordinate (p36_25) at (1.996,-0.0);
\coordinate (p25_46) at (-2.944,-1.7);
\coordinate (p46_25) at (1.0,0.577);
\coordinate (p35_46) at (-2.5,-4.33);
\coordinate (p46_35) at (0.333,0.577);

\filldraw (p13_24) circle (2pt) node[right] {$13,24$};
\filldraw (p24_13) circle (2pt);
\filldraw (p13_25) circle (2pt) node[above right] {$13,25$};
\filldraw (p25_13) circle (2pt) node[below left] {$25,13$};
\filldraw (p13_26) circle (2pt) node[above right] {$13,26$};
\filldraw (p26_13) circle (2pt);
\filldraw (p14_25) circle (2pt) node[above right] {$14,25$};
\filldraw (p25_14) circle (2pt) node[below left] {$25,14$};
\filldraw (p14_26) circle (2pt) node[above] {$14,26$};
\filldraw (p26_14) circle (2pt);
\node at ($(p26_14) + (0,-0.4)$) {$26,14$};
\filldraw (p15_26) circle (2pt) node[above left] {$15,26$};
\filldraw (p26_15) circle (2pt);
\filldraw (p14_35) circle (2pt);
\node at ($(p14_35) + (0,0.3)$) {$14,35$};
\filldraw (p35_14) circle (2pt);
\node at ($(p35_14) + (0,-0.3)$) {$35,14$};
\filldraw (p14_36) circle (2pt) node[above left] {$14,36$};
\filldraw (p36_14) circle (2pt) node[below right] {$36,14$};
\filldraw (p15_36) circle (2pt) node[above left] {$15,36$};
\filldraw (p36_15) circle (2pt) node[below right] {$36,15$};
\filldraw (p15_46) circle (2pt) node[left] {$15,46$};
\filldraw (p46_15) circle (2pt);
\filldraw (p24_35) circle (2pt);
\filldraw (p35_24) circle (2pt) node[below right] {$35,24$};
\filldraw (p24_36) circle (2pt) node[above left] {$24,36$};
\filldraw (p36_24) circle (2pt) node[below right] {$36,24$};
\filldraw (p25_36) circle (2pt) node[left] {$25,36$};
\filldraw (p36_25) circle (2pt) node[right] {$36,25$};
\filldraw (p25_46) circle (2pt) node[below left] {$25,46$};
\filldraw (p46_25) circle (2pt) node[above right] {$46,25$};
\filldraw (p35_46) circle (2pt) node[below left] {$35,46$};
\filldraw (p46_35) circle (2pt);

\draw (p13_24) -- (p13_25);
\draw (p13_24) -- (p13_26);
\draw (p13_24) -- (p35_24);
\draw (p13_24) -- (p36_24);
\draw (p24_13) -- (p25_13);
\draw (p24_13) -- (p26_13);
\draw (p24_13) -- (p24_35);
\draw (p24_13) -- (p24_36);
\draw (p13_25) -- (p13_26);
\draw (p13_25) -- (p14_25);
\draw (p13_25) -- (p36_25);
\draw (p25_13) -- (p26_13);
\draw (p25_13) -- (p25_14);
\draw (p25_13) -- (p25_36);
\draw (p13_26) -- (p14_26);
\draw (p13_26) -- (p15_26);
\draw (p26_13) -- (p26_14);
\draw (p26_13) -- (p26_15);
\draw (p14_25) -- (p14_26);
\draw (p14_25) -- (p14_35);
\draw (p14_25) -- (p46_25);
\draw (p25_14) -- (p26_14);
\draw (p25_14) -- (p35_14);
\draw (p25_14) -- (p25_46);
\draw (p14_26) -- (p15_26);
\draw (p14_26) -- (p14_36);
\draw (p26_14) -- (p26_15);
\draw (p26_14) -- (p36_14);
\draw (p15_26) -- (p15_36);
\draw (p15_26) -- (p15_46);
\draw (p26_15) -- (p36_15);
\draw (p26_15) -- (p46_15);
\draw (p14_35) -- (p14_36);
\draw (p14_35) -- (p24_35);
\draw (p14_35) -- (p46_35);
\draw (p35_14) -- (p36_14);
\draw (p35_14) -- (p35_24);
\draw (p35_14) -- (p35_46);
\draw (p14_36) -- (p15_36);
\draw (p14_36) -- (p24_36);
\draw (p36_14) -- (p36_15);
\draw (p36_14) -- (p36_24);
\draw (p15_36) -- (p15_46);
\draw (p15_36) -- (p25_36);
\draw (p36_15) -- (p46_15);
\draw (p36_15) -- (p36_25);
\draw (p15_46) -- (p25_46);
\draw (p15_46) -- (p35_46);
\draw (p46_15) -- (p46_25);
\draw (p46_15) -- (p46_35);
\draw (p24_35) -- (p24_36);
\draw (p24_35) -- (p46_35);
\draw (p35_24) -- (p36_24);
\draw (p35_24) -- (p35_46);
\draw (p24_36) -- (p25_36);
\draw (p36_24) -- (p36_25);
\draw (p25_36) -- (p25_46);
\draw (p36_25) -- (p46_25);
\draw (p25_46) -- (p35_46);
\draw (p46_25) -- (p46_35);
\end{scope}
\end{tikzpicture}
\footnotesize\\
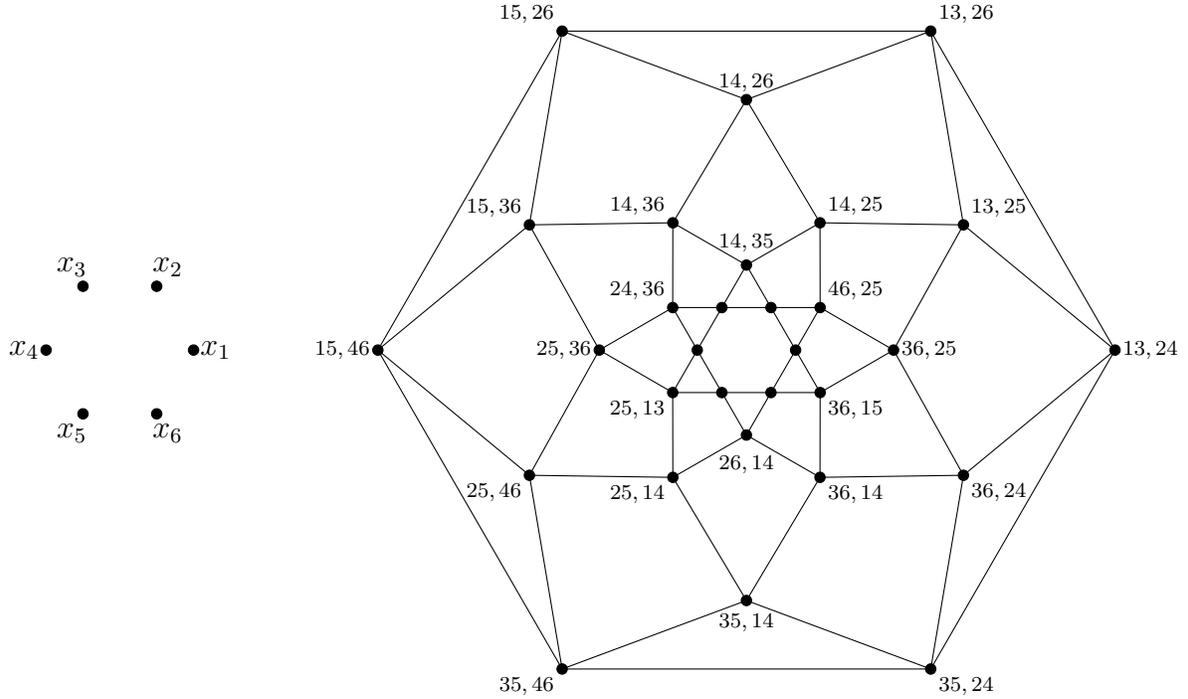
Figure \arabic{figure_num}: points $x_1,\ldots,x_6\in\bbR^2$ (left), and the graph of the corresponding Radon polytope (right).\\The Radon partition $A,B=135,246$ is a partition with tolerance, and corresponds to the hexagonal outer face. Note that the innermost vertex labels have been omitted.
\end{center}

\begin{center}
\refstepcounter{figure_num}\label{no_tolerance_fig}
\begin{tikzpicture}[scale=0.98]
\coordinate (v0) at (360 * 0 / 5:1);
\filldraw (v0) circle (2pt);
\coordinate (v1) at (360 * 1 / 5:1);
\filldraw (v1) circle (2pt);
\coordinate (v2) at (360 * 2 / 5:1);
\filldraw (v2) circle (2pt);
\coordinate (v3) at (360 * 3 / 5:1);
\filldraw (v3) circle (2pt);
\coordinate (v4) at (360 * 4 / 5:1);
\filldraw (v4) circle (2pt);
\coordinate (v5) at (0,0) node[above] {$x_6$};
\filldraw (v5) circle (2pt);

\node at (360 * 0 / 5:1.3) {$x_1$};
\node at (360 * 1 / 5:1.3) {$x_2$};
\node at (360 * 2 / 5:1.3) {$x_3$};
\node at (360 * 3 / 5:1.3) {$x_4$};
\node at (360 * 4 / 5:1.3) {$x_5$};

\begin{scope}[shift={(8.5,0)}]\scriptsize
\coordinate (p13_24) at (-1.172,1.894);
\coordinate (p24_13) at (1.439,-1.699);
\coordinate (p13_25) at (-0.165,2.221);
\coordinate (p25_13) at (-0.165,-2.221);
\coordinate (p13_26) at (-0.963,2.965);
\coordinate (p26_13) at (0.428,-1.317);
\coordinate (p14_25) at (1.439,1.699);
\coordinate (p25_14) at (-1.172,-1.894);
\coordinate (p124_6) at (5.0,0.0);
\coordinate (p6_124) at (-0.682,0.0);
\coordinate (p16_25) at (0.428,1.317);
\coordinate (p25_16) at (-0.963,-2.965);
\coordinate (p14_35) at (2.061,0.844);
\coordinate (p35_14) at (-2.163,-0.529);
\coordinate (p134_6) at (1.545,4.755);
\coordinate (p6_134) at (-0.211,-0.649);
\coordinate (p135_6) at (-4.045,2.939);
\coordinate (p6_135) at (0.552,-0.401);
\coordinate (p14_56) at (2.522,1.833);
\coordinate (p56_14) at (-1.12,-0.814);
\coordinate (p24_35) at (2.061,-0.844);
\coordinate (p35_24) at (-2.163,0.529);
\coordinate (p24_36) at (2.522,-1.833);
\coordinate (p36_24) at (-1.12,0.814);
\coordinate (p235_6) at (-4.045,-2.939);
\coordinate (p6_235) at (0.552,0.401);
\coordinate (p245_6) at (1.545,-4.755);
\coordinate (p6_245) at (-0.211,0.649);
\coordinate (p35_46) at (-3.118,0.0);
\coordinate (p46_35) at (1.385,0.0);

\filldraw (p13_24) circle (2pt) node[above left] {$13,24$};
\filldraw (p24_13) circle (2pt);
\node at ($(p24_13) + (0,-0.275)$) {$24,13$};
\filldraw (p13_25) circle (2pt);
\node at ($(p13_25) + (0.4,0.275)$) {$13,25$};
\filldraw (p25_13) circle (2pt);
\node at ($(p25_13) + (0.4,-0.275)$) {$25,13$};
\filldraw (p13_26) circle (2pt) node[above left] {$13,26$};
\filldraw (p26_13) circle (2pt) node[above right] {$26,13$};
\filldraw (p14_25) circle (2pt) node[above] {$14,25$};
\filldraw (p25_14) circle (2pt) node[left] {$25,14$};
\filldraw (p124_6) circle (2pt) node[right] {$124,6$};
\filldraw (p6_124) circle (2pt);
\filldraw (p16_25) circle (2pt) node[left] {$16,25$};
\filldraw (p25_16) circle (2pt);
\node at ($(p25_16) + (-0.4,-0.275)$) {$25,16$};
\filldraw (p14_35) circle (2pt) node[right] {$14,35$};
\filldraw (p35_14) circle (2pt) node[below left] {$35,14$};
\filldraw (p134_6) circle (2pt) node[above right] {$134,6$};
\filldraw (p6_134) circle (2pt);
\filldraw (p135_6) circle (2pt) node[above left] {$135,6$};
\filldraw (p6_135) circle (2pt);
\filldraw (p14_56) circle (2pt) node[above right] {$14,56$};
\filldraw (p56_14) circle (2pt) node[below right] {$56,14$};
\filldraw (p24_35) circle (2pt) node[right] {$24,35$};
\filldraw (p35_24) circle (2pt) node[above left] {$35,24$};
\filldraw (p24_36) circle (2pt) node[below right] {$24,36$};
\filldraw (p36_24) circle (2pt);
\node at ($(p36_24) + (-0.24,-0.325)$) {$36,24$};
\filldraw (p235_6) circle (2pt) node[below left] {$235,6$};
\filldraw (p6_235) circle (2pt);
\filldraw (p245_6) circle (2pt) node[below right] {$245,6$};
\filldraw (p6_245) circle (2pt);
\filldraw (p35_46) circle (2pt) node[left] {$35,46$};
\filldraw (p46_35) circle (2pt);
\node at ($(p46_35) + (-0.15,0.285)$) {$46,35$};

\draw (p13_24) -- (p13_25);
\draw (p13_24) -- (p13_26);
\draw (p13_24) -- (p35_24);
\draw (p13_24) -- (p36_24);
\draw (p24_13) -- (p25_13);
\draw (p24_13) -- (p26_13);
\draw (p24_13) -- (p24_35);
\draw (p24_13) -- (p24_36);
\draw (p13_25) -- (p13_26);
\draw (p13_25) -- (p14_25);
\draw (p13_25) -- (p16_25);
\draw (p25_13) -- (p26_13);
\draw (p25_13) -- (p25_14);
\draw (p25_13) -- (p25_16);
\draw (p13_26) -- (p134_6);
\draw (p13_26) -- (p135_6);
\draw (p26_13) -- (p6_134);
\draw (p26_13) -- (p6_135);
\draw (p14_25) -- (p16_25);
\draw (p14_25) -- (p14_35);
\draw (p14_25) -- (p14_56);
\draw (p25_14) -- (p25_16);
\draw (p25_14) -- (p35_14);
\draw (p25_14) -- (p56_14);
\draw (p124_6) -- (p134_6);
\draw (p124_6) -- (p14_56);
\draw (p124_6) -- (p24_36);
\draw (p124_6) -- (p245_6);
\draw (p6_124) -- (p6_134);
\draw (p6_124) -- (p56_14);
\draw (p6_124) -- (p36_24);
\draw (p6_124) -- (p6_245);
\draw (p16_25) -- (p6_235);
\draw (p16_25) -- (p6_245);
\draw (p25_16) -- (p235_6);
\draw (p25_16) -- (p245_6);
\draw (p14_35) -- (p14_56);
\draw (p14_35) -- (p24_35);
\draw (p14_35) -- (p46_35);
\draw (p35_14) -- (p56_14);
\draw (p35_14) -- (p35_24);
\draw (p35_14) -- (p35_46);
\draw (p134_6) -- (p135_6);
\draw (p134_6) -- (p14_56);
\draw (p6_134) -- (p6_135);
\draw (p6_134) -- (p56_14);
\draw (p135_6) -- (p235_6);
\draw (p135_6) -- (p35_46);
\draw (p6_135) -- (p6_235);
\draw (p6_135) -- (p46_35);
\draw (p24_35) -- (p24_36);
\draw (p24_35) -- (p46_35);
\draw (p35_24) -- (p36_24);
\draw (p35_24) -- (p35_46);
\draw (p24_36) -- (p245_6);
\draw (p36_24) -- (p6_245);
\draw (p235_6) -- (p245_6);
\draw (p235_6) -- (p35_46);
\draw (p6_235) -- (p6_245);
\draw (p6_235) -- (p46_35);
\end{scope}
\end{tikzpicture}
\footnotesize\\
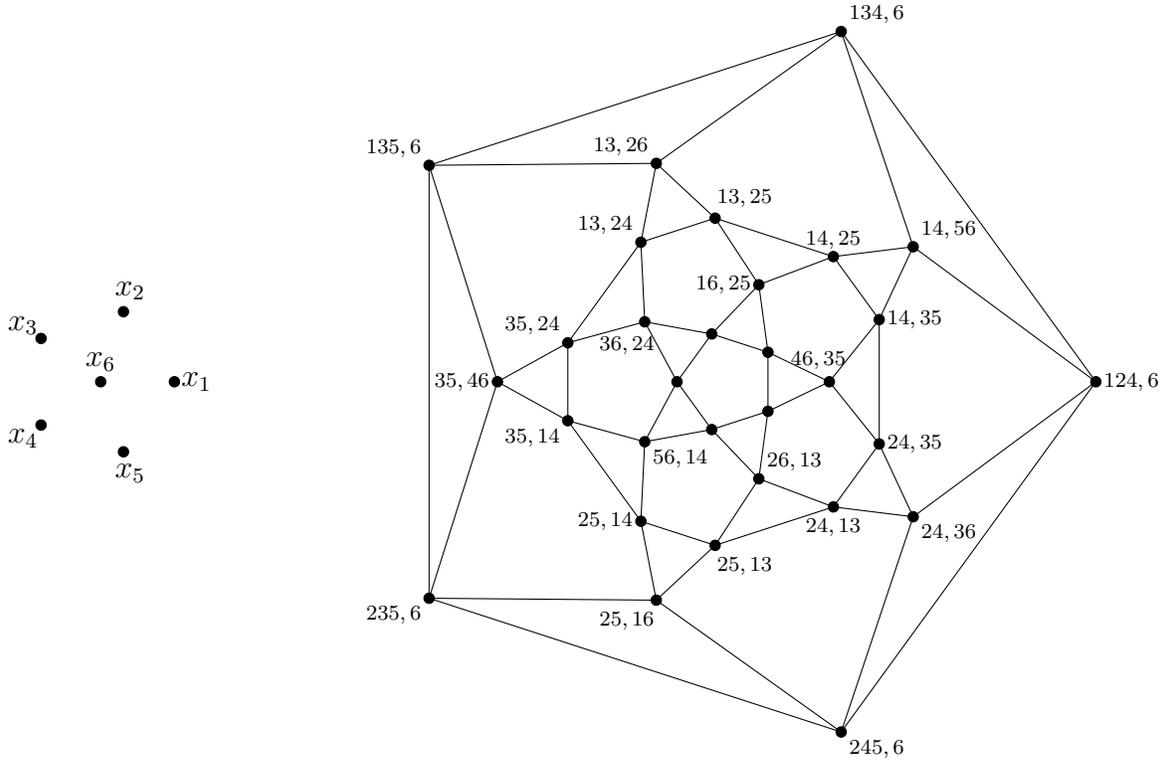
Figure \arabic{figure_num}: points $x_1,\ldots,x_6\in\bbR^2$ (left), and the graph of the corresponding Radon polytope (right).\\These points have no partition with tolerance, and the polytope has no hexagonal face. Note that the innermost vertex labels have been omitted.
\end{center}

\subsection{Reay's relaxed Tverberg conjecture}
Consider the following problem: given $N$ points in $\Rd$, can we find \emph{three} disjoint sets $A,B,C\su[N]$, such that $(A,B),\ (A,C),\ (B,C)$ are all Radon partitions? We call $(A,B,C)$ a Reay partition. Let $T(d,3,2)$ denote the smallest $N$ such that \emph{every} collection of $N$ points in $\Rd$ has a Reay partition. This is one case in a problem introduced by Reay in 1979, who conjectured that $T(d,3,2)=2d+3$ (in general $T(d,r,k)$ is the minimal $N$ such that every set of $N$ points in $\Rd$ can be partitioned into $r$ disjoint sets, such that the convex hull of every $k$ among them has non-empty intersection, see \cite{Re}). It follows from Tverberg's theorem that $T(d,3,2)\leq 2d+3$, and equality was proved for $d\leq 5$ by Perles and Sigron \cite{PS}.

Since Reay's conjecture can be formulated using Radon partitions, it can be understood as a question about Radon polytopes:
\begin{ques}
Let $d\in\bbN$ and denote $N=2d+2$. Does there exist a linear subspace $V\subset\An$ of dimension $d+1$, such that for every disjoint $A,B,C\subset[N]$, at least one of the intersections $V\cap F(A,B),\ V\cap F(A,C),\ V\cap F(B,C)$ is empty?
\end{ques}
For every $d$, the answer to the above is positive if and only if Reay's conjecture holds for $(d,3,2)$.

It seems possible to obtain a kinematic formula for multiple cones, generalizing (\ref{kinematic_formula}) and computing the probability that random collection of $N$ points in $\Rd$ obtains a certain Reay partition. For $6$ points, simulations allow us to estimate
$$\bbP(6\text{ points in }\bbR^2\text{ have a Reay partition})\approx 0.427$$
(more accurately, the probability is $0.42714\pm0.00016$ with $99.7\%$ confidence, by simulating $82\cdot10^6$ times).

\section{Appendix - some values of $v_k(m,n)$}
We list some closed formulas and some numeric approximations of $v_k(m,n)$, for $m=1,2,3$, as well as for $k=m+n-1$, in the tables below. In this section, $a_j=\arcsin\frac1j$, a blank value indicates $0$, and ``?'' indicates no known formula. W.l.o.g assume $m\leq n$.

For closed formulas, the case $k=0$ follows from Corollary \ref{v_0}. In general we apply (\ref{formula_m1}), (\ref{formula_m2}), and (\ref{formula_m3}), where Proposition \ref{g_values} 1-4 directly applies if only values of $g_\ell$ for $\ell\leq 3$ are involved. We then deduce the value of $v_4(1,4)=g_4(-\frac15)$ from (\ref{gauss-bonnet}), which allows us to compute a few more values.

\subsection{The case $m=1$}

Some closed form values of $v_k(1,n)$ are:
\[\def\arraystretch{1.4}
\begin{array}{|c|c|c|c|c|c|c|c|c|}
\hline
\raisebox{-.2em}{$k$} \diagdown \raisebox{.2em}{$n$} &1 &2 &3 &4 &5 &6 &7 &8 \\
\hline
0 &\frac12 &\frac13 &\frac14 &\frac15 &\frac16 &\frac17 &\frac18 &\frac19 \\
\hline
1 &\frac12 &\frac12 &\frac38+\frac3{4\pi}\ar3 &\frac14+\frac3{2\pi}\ar3 &? &? &? &? \\
\hline
2 & &\frac16 &\frac14 &\frac14+\frac1{2\pi}\ar4 &\frac5{24}+\frac5{4\pi}\ar4 &? &? &? \\
\hline
3 & & &\frac18-\frac3{4\pi}\ar3 &\frac14-\frac3{2\pi}\ar3 &
\begin{split}
\textstyle\Par{\frac58-\frac{15}{4\pi}\ar3}\cdot&\\
\textstyle\Par{\frac12+\frac1{\pi}\ar5}&
\end{split} &
\begin{split}
\textstyle\Par{\frac58-\frac{15}{4\pi}\ar3}\cdot&\\
\textstyle\Par{\frac12+\frac3{\pi}\ar5}&
\end{split} &? &? \\
\hline
4 & & & &\frac1{20}-\frac1{2\pi}\ar4 &\frac18-\frac5{4\pi}\ar4 &
\begin{split}
\textstyle\Par{\frac14-\frac5{2\pi}\ar4}\cdot&\\
\textstyle\Par{\frac34+\frac3{2\pi}\ar6}&
\end{split} &
\begin{split}
\textstyle\Par{\frac14-\frac5{2\pi}\ar4}\cdot&\\
\textstyle\Par{\frac78+\frac{21}{4\pi}\ar6}&
\end{split} &? \\
\hline
5 & & & & &? &? &? &? \\
\hline
\end{array}\]
With approximate numerical values:
\[\def\arraystretch{1.6}\scriptsize\begin{array}{|c|c|c|c|c|c|c|c|c|c|}
\hline
\normalsize{\raisebox{-.2em}{$k$} \diagdown \raisebox{.2em}{$n$}} &$\normalsize{1}$ &$\normalsize{2}$ &$\normalsize{3}$ &$\normalsize{4}$ &$\normalsize{5}$ &$\normalsize{6}$ &$\normalsize{7}$ &$\normalsize{8}$ &$\normalsize{9}$ \\
\hline
$\normalsize{0}$ &$\small{.5}$ &$\small{.3333}$ &$\small{.25}$ &$\small{.2}$ &$\small{.1667}$ &$\small{.1429}$ &$\small{.125}$ &$\small{.1111}$ &$\small{.1}$ \\
\hline
$\normalsize{1}$ &$\small{.5}$ &$\small{.5}$ &$\small{.4561}$ &$\small{.4123}$ &$\small{.3743}$ &$\small{.3424}$ &$\small{.3154}$ &$\small{.2925}$ &$\small{.2727}$ \\
\hline
$\normalsize{2}$ & &$\small{.1667}$ &$\small{.25}$ &$\small{.2902}$ &$\small{.3089}$ &$\small{.3162}$ &$\small{.3173}$ &$\small{.3149}$ &$\small{.3105}$ \\
\hline
$\normalsize{3}$ & & &$\small{.04387}$ &$\small{.08774}$ &$\small{.1237}$ &$\small{.1519}$ &$\small{.1735}$ &$\small{.1902}$ &$\small{.203}$ \\
\hline
$\normalsize{4}$ & & & &$\small{.009785}$ &$\small{.02446}$ &$\small{.0406}$ &$\small{.0565}$ &$\small{.07139}$ &$\small{.08499}$ \\
\hline
$\normalsize{5}$ & & & & &$\small{.001922}$ &$\small{.005766}$ &$\small{.01101}$ &$\small{.01714}$ &$\small{.02374}$ \\
\hline
$\normalsize{6}$ & & & & & &3.407\cdot 10^{-4} &$\small{.001192}$ &$\small{.002575}$ &$\small{.004434}$ \\
\hline
$\normalsize{7}$ & & & & & & &5.541\cdot 10^{-5} &2.216\cdot 10^{-4} &5.34\cdot 10^{-4} \\
\hline
$\normalsize{8}$ & & & & & & & &8.369\cdot 10^{-6} &3.766\cdot 10^{-5} \\
\hline
$\normalsize{9}$ & & & & & & & & &1.185\cdot 10^{-6} \\
\hline
\end{array}\]

\subsection{The case $m=2$}

Some closed form values of $v_k(2,n)$ are:
\[\def\arraystretch{1.4}
\begin{array}{|c|c|c|c|c|c|c|}
\hline
\raisebox{-.2em}{$k$} \diagdown \raisebox{.2em}{$n$} &2 &3 &4 &5 &6 &7 \\
\hline
0 &\frac16 &\frac1{10} &\frac1{15} &\frac1{21} &\frac1{28} &\frac1{36} \\
\hline
1 &\frac12-\frac1{\pi}\ar3 &\frac38-\frac3{4\pi}\ar3 &? &? &? &? \\
\hline
2 &\frac13 &\frac38-\frac1{4\pi}\ar4 &\frac13 &? &? &? \\
\hline
3 &\frac1{\pi}\ar3 &\frac18+\frac3{4\pi}\ar3 &
\begin{split}
&\textstyle\frac14-\frac1{2\pi}\ar5\\
&\textstyle+\frac6{\pi^2}\ar3\ar5
\end{split} &
\begin{split}
&\textstyle\frac5{16}-\frac5{8\pi}\ar3\\
&\textstyle-\frac{5}{8\pi}\ar5+\frac{45}{4\pi^2}\ar3\ar5 
\end{split} &? &? \\
\hline
4 & &\frac1{40}+\frac1{4\pi}\ar4 &\frac1{10} &
\begin{split}
&\textstyle\frac3{16}-\frac5{8\pi}\ar4\\
&\textstyle-\frac1{8\pi}\ar6+\frac{15}{4\pi^2}\ar4\ar6
\end{split} &
\begin{split}
&\textstyle\frac14-\frac5{4\pi}\ar4 \\
&\textstyle+\frac{30}{4\pi^2}\ar4\ar6
\end{split} &? \\
\hline
5 & & &? &? &? &? \\
\hline
\end{array}\]
With approximate numerical values:
\[\def\arraystretch{1.6}\scriptsize\begin{array}{|c|c|c|c|c|c|c|c|c|}
\hline
\normalsize{\raisebox{-.2em}{$k$} \diagdown \raisebox{.2em}{$n$}} &$\normalsize{2}$ &$\normalsize{3}$ &$\normalsize{4}$ &$\normalsize{5}$ &$\normalsize{6}$ &$\normalsize{7}$ &$\normalsize{8}$ &$\normalsize{9}$ \\
\hline
$\normalsize{0}$ &$\small{.1667}$ &$\small{.1}$ &$\small{.06667}$ &$\small{.04762}$ &$\small{.03571}$ &$\small{.02778}$ &$\small{.02222}$ &$\small{.01818}$ \\
\hline
$\normalsize{1}$ &$\small{.3918}$ &$\small{.2939}$ &$\small{.2256}$ &$\small{.1781}$ &$\small{.1441}$ &$\small{.119}$ &$\small{.1}$ &$\small{.08535}$ \\
\hline
$\normalsize{2}$ &$\small{.3333}$ &$\small{.3549}$ &$\small{.3333}$ &$\small{.3015}$ &$\small{.2698}$ &$\small{.241}$ &$\small{.2158}$ &$\small{.194}$ \\
\hline
$\normalsize{3}$ &$\small{.1082}$ &$\small{.2061}$ &$\small{.2596}$ &$\small{.2828}$ &$\small{.2888}$ &$\small{.2852}$ &$\small{.2766}$ &$\small{.2655}$ \\
\hline
$\normalsize{4}$ & &$\small{.04511}$ &$\small{.1}$ &$\small{.1466}$ &$\small{.1816}$ &$\small{.2062}$ &$\small{.2225}$ &$\small{.2326}$ \\
\hline
$\normalsize{5}$ & & &$\small{.01488}$ &$\small{.03912}$ &$\small{.06612}$ &$\small{.09208}$ &$\small{.1153}$ &$\small{.1351}$ \\
\hline
$\normalsize{6}$ & & & &$\small{.004191}$ &$\small{.01291}$ &$\small{.02486}$ &$\small{.03856}$ &$\small{.05287}$ \\
\hline
$\normalsize{7}$ & & & & &$\small{.001049}$ &$\small{.003728}$ &$\small{.008072}$ &$\small{.01382}$ \\
\hline
$\normalsize{8}$ & & & & & &2.392\cdot 10^{-4} &9.65\cdot 10^{-4} &$\small{.002325}$ \\
\hline
$\normalsize{9}$ & & & & & & &5.042\cdot 10^{-5} &2.281\cdot 10^{-4} \\
\hline
$\normalsize{10}$ & & & & & & & &9.943\cdot 10^{-6} \\
\hline
\end{array}\]

\subsection{The case $m=3$}

Some closed form values of $v_k(3,n)$ are:
\[\def\arraystretch{1.4}
\begin{array}{|c|c|c|c|c|c|c|}
\hline
\raisebox{-.2em}{$k$} \diagdown \raisebox{.2em}{$n$} &3 &4 &5 &6 \\
\hline
0 &\frac1{20} &\frac1{35} &\frac1{56} &\frac1{84} \\
\hline
1 &? &? &? &? \\
\hline
2 &\frac38-\frac3{4\pi}\ar4 &? &? &? \\
\hline
3 &
\begin{split}
&\textstyle\frac3{16}+\frac9{8\pi}\ar3\\
&\textstyle+\frac3{8\pi}\ar5-\frac{27}{4\pi^2}\ar3\ar5
\end{split} &
\begin{split}
&\textstyle\frac14+\frac3{4\pi}\ar3\\
&\textstyle-\frac9{2\pi^2}\ar3\ar5
\end{split} &? &? \\
\hline
4 &\frac3{40}+\frac3{4\pi}\ar4 &\frac3{20}+\frac3{4\pi}\ar4-\frac3{2\pi^2}\ar4\ar6 &
\begin{split}
&\textstyle\frac7{32}+\frac5{16\pi}\ar4\\
&\textstyle-\frac3{16\pi}\ar6+\frac{15}{8\pi^2}\ar4\ar6
\end{split} &? \\
\hline
5 &? &? &? &? \\
\hline
\end{array}\]
With approximate numerical values:
\[\def\arraystretch{1.6}\scriptsize\begin{array}{|c|c|c|c|c|c|c|c|}
\hline
\normalsize{\raisebox{-.2em}{$k$} \diagdown \raisebox{.2em}{$n$}} &$\normalsize{3}$ &$\normalsize{4}$ &$\normalsize{5}$ &$\normalsize{6}$ &$\normalsize{7}$ &$\normalsize{8}$ &$\normalsize{9}$ \\
\hline
$\normalsize{0}$ &$\small{.05}$ &$\small{.02857}$ &$\small{.01786}$ &$\small{.0119}$ &$\small{.008333}$ &$\small{.006061}$ &$\small{.004545}$ \\
\hline
$\normalsize{1}$ &$\small{.187}$ &$\small{.1247}$ &$\small{.087}$ &$\small{.06308}$ &$\small{.04721}$ &$\small{.03628}$ &$\small{.0285}$ \\
\hline
$\normalsize{2}$ &$\small{.3147}$ &$\small{.2558}$ &$\small{.2047}$ &$\small{.1643}$ &$\small{.1333}$ &$\small{.1093}$ &$\small{.09071}$ \\
\hline
$\normalsize{3}$ &$\small{.2864}$ &$\small{.2999}$ &$\small{.2841}$ &$\small{.258}$ &$\small{.2302}$ &$\small{.2039}$ &$\small{.1803}$ \\
\hline
$\normalsize{4}$ &$\small{.1353}$ &$\small{.2039}$ &$\small{.2419}$ &$\small{.2574}$ &$\small{.2586}$ &$\small{.2516}$ &$\small{.2402}$ \\
\hline
$\normalsize{5}$ &$\small{.02653}$ &$\small{.07538}$ &$\small{.1247}$ &$\small{.1645}$ &$\small{.193}$ &$\small{.2114}$ &$\small{.222}$ \\
\hline
$\normalsize{6}$ & &$\small{.0117}$ &$\small{.03554}$ &$\small{.06502}$ &$\small{.09465}$ &$\small{.1214}$ &$\small{.1439}$ \\
\hline
$\normalsize{7}$ & & &$\small{.00428}$ &$\small{.01441}$ &$\small{.02923}$ &$\small{.04672}$ &$\small{.06512}$ \\
\hline
$\normalsize{8}$ & & & &$\small{.001365}$ &$\small{.005137}$ &$\small{.01152}$ &$\small{.02013}$ \\
\hline
$\normalsize{9}$ & & & & &3.915\cdot 10^{-4} &$\small{.001641}$ &$\small{.004049}$ \\
\hline
$\normalsize{10}$ & & & & & &1.029\cdot 10^{-4} &4.78\cdot 10^{-4} \\
\hline
$\normalsize{11}$ & & & & & & &2.514\cdot 10^{-5} \\
\hline
\end{array}\]

\subsection{The case $k=m+n-1$}

Approximate values of $v_{m+n-1}(m,n)$ (where $m\leq n$) are given by:
\[\def\arraystretch{1.6}\scriptsize\begin{array}{|c|c|c|c|c|c|c|c|c|c|}
\hline
\normalsize{\raisebox{-.2em}{$m$} \diagdown \raisebox{.2em}{$n$}} &$\normalsize{2}$ &$\normalsize{3}$ &$\normalsize{4}$ &$\normalsize{5}$ &$\normalsize{6}$ &$\normalsize{7}$ &$\normalsize{8}$ &$\normalsize{9}$ &$\normalsize{10}$ \\
\hline
$\normalsize{1}$ &$\small{.1667}$ &$\small{.04387}$ &$\small{.009785}$ &$\small{.001922}$ &3.407\cdot 10^{-4} &5.541\cdot 10^{-5} &8.369\cdot 10^{-6} &1.185\cdot 10^{-6} &1.583\cdot 10^{-7} \\
\hline
$\normalsize{2}$ &$\small{.1082}$ &$\small{.04511}$ &$\small{.01488}$ &$\small{.004191}$ &$\small{.001049}$ &2.392\cdot 10^{-4} &5.042\cdot 10^{-5} &9.943\cdot 10^{-6} &1.85\cdot 10^{-6} \\
\hline
$\normalsize{3}$ & &$\small{.02653}$ &$\small{.0117}$ &$\small{.00428}$ &$\small{.001365}$ &3.915\cdot 10^{-4} &1.029\cdot 10^{-4} &2.514\cdot 10^{-5} &5.766\cdot 10^{-6} \\
\hline
$\normalsize{4}$ & & &$\small{.006586}$ &$\small{.002989}$ &$\small{.001162}$ &4.011\cdot 10^{-4} &1.257\cdot 10^{-4} &3.636\cdot 10^{-5} &9.82\cdot 10^{-6} \\
\hline
$\normalsize{5}$ & & & &$\small{.00164}$ &7.578\cdot 10^{-4} &3.07\cdot 10^{-4} &1.119\cdot 10^{-4} &3.737\cdot 10^{-5} &1.159\cdot 10^{-5} \\
\hline
$\normalsize{6}$ & & & & &4.091\cdot 10^{-4} &1.913\cdot 10^{-4} &7.981\cdot 10^{-5} &3.029\cdot 10^{-5} &1.061\cdot 10^{-5} \\
\hline
$\normalsize{7}$ & & & & & &1.021\cdot 10^{-4} &4.818\cdot 10^{-5} &2.055\cdot 10^{-5} &8.046\cdot 10^{-6} \\
\hline
$\normalsize{8}$ & & & & & & &2.55\cdot 10^{-5} &1.211\cdot 10^{-5} &5.256\cdot 10^{-6} \\
\hline
$\normalsize{9}$ & & & & & & & &6.369\cdot 10^{-6} &3.042\cdot 10^{-6} \\
\hline
$\normalsize{10}$ & & & & & & & & &1.591\cdot 10^{-6} \\
\hline
\end{array}\]

\bigskip

\bibliographystyle{plain}

\end{document}